\documentclass[12pt,english]{amsart}

 \usepackage[paperwidth=210mm,paperheight=297mm,inner=3cm,outer=3cm,top=2.5cm,bottom=2.6cm,marginparwidth=1.8cm]{geometry}

\usepackage[T1]{fontenc}
\usepackage{amssymb}

\renewcommand\emph{\textbf}

\usepackage[leqno]{amsmath}
\usepackage[pdftex]{graphicx}    
\usepackage{amsthm}
\usepackage{amsrefs}
\usepackage{babel}
\usepackage{caption}
\usepackage{etex}
\usepackage{fancyvrb}
\usepackage[colorlinks,linkcolor=black,citecolor=black,urlcolor=black,hypertexnames=true]{hyperref}
\usepackage[retainorgcmds]{IEEEtrantools}
\usepackage{paralist}
\usepackage{tikz}
\usetikzlibrary{shapes,patterns,calc,snakes}
\usepackage{xypic}

\newcommand\bracemap[4]{
\left\{
  \begin{array}{ccc}
    #1 & \to & #2\\
    #3 & \mapsto & #4
  \end{array}\right.
}

\newcommand\mbb{\mathbb}

\newcommand\mcal{\mathcal}

\newcommand\ol{\overline}

\newcommand\sH{\mcal{H}}

\newcommand\C{\mbb{C}}
\newcommand\D{\mbb{D}}

\newcommand\R{\mbb{R}}
\newcommand\T{\mbb{T}}

\newcommand\scp[1]{\langle #1\rangle}
\newcommand\bigscp[1]{\bigl\langle #1\bigl\rangle}

\renewcommand\epsilon{\varepsilon}
\renewcommand\ge{\geqslant}

\renewcommand\phi{\varphi}

\newcommand\st{\:|\:}
\renewcommand\theta{\vartheta}

\DeclareMathOperator\Sym{Sym}

\DeclareMathOperator\id{id}

\numberwithin{equation}{section}

\newtheoremstyle{smallthm}
{1em}
{1em}
{\small}
{0pt}
{\small\bfseries}
{}
{0pt}
{\thmname{#1}\thmnumber{ #2}\thmnote{ #3}.\enspace}

\theoremstyle{remark}

\theoremstyle{plain}
\newtheorem{Thm}[equation]{Theorem}
\newtheorem{Prop}[equation]{Proposition}
\newtheorem{Cor}[equation]{Corollary}
\newtheorem{Lemma}[equation]{Lemma}

\newtheorem*{Thm*}{Theorem}
\newtheorem*{Prop*}{Proposition}
\newtheorem*{Cor*}{Corollary}
\newtheorem*{Lemma*}{Lemma}
\newtheorem*{Sublemma*}{Sublemma}
\newtheorem*{Conjecture*}{Conjecture}

\theoremstyle{definition}
\newtheorem{Block}[equation]{}
\newtheorem{Def}[equation]{Definition}

\newtheorem{Example}[equation]{Example}

\newtheorem{Remark}[equation]{Remark}

\newtheorem*{Def*}{Definition}
\newtheorem*{Defs*}{Definitions}
\newtheorem*{Example*}{Example}
\newtheorem*{Examples*}{Examples}
\newtheorem*{LemmaDef*}{Lemma and Definition}
\newtheorem*{Notation*}{Notation}
\newtheorem*{Problem*}{Problem}
\newtheorem*{Question*}{Question}
\newtheorem*{Remark*}{Remark}
\newtheorem*{Remarks*}{Remarks}
\newtheorem*{Warning*}{Warning}

\theoremstyle{smallthm}

\newtheorem*{Exercise*}{Exercise}
\numberwithin{Exercise}{section}

\begin{document}

\title[]{A relative Grace Theorem for complex polynomials}

\begin{abstract}
  We study the pullback of the apolarity invariant of complex
  polynomials in one variable under a polynomial map on the complex
  plane. As a consequence, we obtain variations of the classical
  results of Grace and Walsh in which the unit disk, or a circular
  domain, is replaced by its image under the given polynomial map.
\end{abstract}

\author{Daniel Plaumann}
\address{Universit\"at Konstanz}
\email{Daniel.Plaumann@uni-konstanz.de}
\author{Mihai Putinar}
\address{University of California at Santa
  Barbara and Newcastle
  University}
\email{mputinar@math.ucsb.edu}
\email{mihai.putinar@ncl.ac.uk}


\maketitle

\section*{Introduction}

Let $f(z) = a_0 + a_1 z + \ldots + a_n z^n$ and $g(z) = b_0 + b_1 z + \ldots + b_n z^n$ be
two polynomials with complex coefficients and degrees less than or equal to $n$. The only
joint invariant under affine substitutions which is linear in the coefficients is
\[
[f,g]_n = a_0 b_n - \frac{1}{n} a_1 b_{n-1} + \frac{1}{\binom{n}{2}}
a_2 b_{n-2} + \ldots + (-1)^n a_n b_0.
\]
The two polynomials are called apolar if $[f,g]_n =0$. Apolarity
provides the ground for the study of the simplest type of invariant,
generalizing in higher degree the notion of harmonic quadrics. The
geometric implications of apolarity are surprising and multifold, see
for instance \cite{GY03}. A celebrated result due to Grace and Heawood
asserts that the complex zeros of two apolar polynomials cannot be
separated by a circle or a straight line. As a matter of fact, this
zero location property is equivalent to apolarity \cite{Di38}. An
array of independent proofs of Grace's theorem, as the result is named
nowadays, are known (see \cite{Ho54, Ma66, RS02, Sz22}). The
common technical ingredient of these proofs is a lemma of Laguerre and
induction. Notable is also the coincidence, up to a conjugation in the
second argument, of the apolarity invariant and Fischer's inner
product:
\[
  \scp{f,g}_n = \sum_{k=0}^n \frac{1}{\binom{n}{k}}a_k\ol{b_k},
\]
well known for identifying the adjoint of complex differentiation with
the multiplication by the variable \cite{Fi17}.
  
It is natural and convenient to symmetrize the polynomials $f$ and $g$
and interpret both apolarity and Fischer inner products in terms of
the roots of these polynomials. In this direction an observation due
to Walsh establishes a powerful equivalent statement to Grace's theorem,
known as the Walsh Coincidence Theorem: If a monic polynomial $f$ of
degree $n$ has no zeros in a circular domain $D$ (that is a disc,
complement of a disk, a half-space, open or closed), then its
symmetrization has no zeros in $D^n$, \cite{Ho54, Ma66, RS02,
  Sz22}. In its turn, Walsh's theorem explains and offers a natural
framework for a series of polynomial inequalities, in one or several
variables \cite{Be26, CS35, Ho54}.
  
The aim of the present note is to study the pull-back of the apolarity
invariant by a polynomial mapping $q$ of degree $d$.  It turns out
that the pull-back form can be represented on the original space of
polynomials of degree $n$ by an invertible linear operator $T_{q,n}$:
\[
[f \circ q, g \circ q]_{nd} = [T_{q,n} f, g]_n.
\]
At the geometric level, the pull back is transformed into push-forward
by the ramified cover map $q$. For a circular domain $D$, we
distinguish between the set theoretic image $q(D)$ and the full
push-forward $q_\circ (D)$ which consists of all points in $q(D)$
having the full fiber contained in $D$. Relative $q$-versions of the
theorems of Grace and Walsh follow easily, for instance: {\it If two
  polynomials $f$ and $g$ are $q-apolar$ and all zeros of $f$ are contained
  in $q_\circ(D)$, then $g$ has a zero in $q(D)$}. As expected, the
$q$-version of the Walsh Coincidence Theorem has non-trivial consequences in
the form of polynomial inequalities for symmetric polynomials of
several variables. To be more specific we prove below that a
majorization on the diagonal of $q(D)$ is
transmitted, with the same constant, to the polydomain $q_\circ(D)^n$.
  
The operator $T_{q,n}$ representing the $q$-apolarity form is complex
symmetric, in the sense of \cite{GPP14}, and as a consequence the
eigenfunctions of its modulus are doubly orthogonal, both in the
Fischer norm, and with respect to the apolarity bilinear form. In
particular the zeros of these polynomials cannot be separated by
higher degree algebraic sets deduced from the boundaries of $q(D)$,
respectively $q_\circ(D)$.
  
The idea of providing a more flexible version of Grace's theorem 
also relates to recent work of B.~and H.~Sendov in \cite{Se14}, which
is concerned with finding minimal domains satisfying the statement of Grace's
theorem for a given polynomial.

This paper is structured as follows: Section \ref{Sec:Prelim} contains
notation, preliminaries and some of the classical results. In Section
\ref{Sec:PolyImages}, we discuss polynomial images of circular
domains. Section \ref{Sec:Symmetrization} contains the core results on
symmetrization and pullback and the relative versions of the theorems
of Grace and Walsh. Section \ref{Sec:SkewEigenfunctions} concerns the
study of
skew-eigenfunctions of the symmetrization operator $T_{q,n}$ in the
sense of \cite{GPP14}. 

\medskip
\textit{Acknowledgements.} We would like to thank one of the referees for
helpful comments. Daniel Plaumann was partially supported by
a Research Fellowship of the Nanyang Technological University.

\section{Preliminaries}\label{Sec:Prelim}
\noindent We write $\C[z]_n$ for the space of complex polynomials in $z$ of degree at most $n$.

\begin{Block}
  For $f=\sum_{k=0}^na_kz^k\in\C[z]_n$, the \emph{symmetrization} of
  $f$ in degree $n$ is the polynomial in $n$ variables $y_1,\dots,y_n$
  defined by
  \[
  \Sym_n(f)=\sum_{k=0}^n
  \frac{a_k}{\binom{n}{k}}\sigma_k(y_1,\dots,y_n),
  \]
  where $\sigma_k(y_1,\dots,y_n)$ is the elementary symmetric
  polynomial of degree $k$ in $y_1,\dots,y_n$. The symmetrization is
  the unique multiaffine symmetric polynomial of degree at most $n$ in
  $y_1,\dots,y_n$ with the property
  \[
  \Sym_n(f)(z,\dots,z)=f(z).
  \]
  To verify uniqueness, just note that a multiaffine symmetric
  polynomial of degree at most $n$ in $y_1,\dots,y_n$ is necessarily of the form
  $\sum_{k=0}^n b_k\sigma_k(y_1,\dots,y_n)$.
\end{Block}

\begin{Block}\label{SharpCheck}
  For $f=\sum_{k=0}^na_kz^k\in\C[z]_n$, we write
  \begin{align*}
    &f^\vee = f(-z) = \sum\nolimits_{k=0}^n (-1)^k a_k z^k\\
    &f^\# = z^n\ol{f\bigl(1/\ol{z}\bigr)} = \sum\nolimits_{k=0}^n
    \ol{a_{n-k}}z^k
  \end{align*}
  Observe that $(fg)^\# = f^\#g^\#$ and $\bigl(\prod_{k=1}^n
  (z-\lambda_k)\bigr)^{\#\vee} = \prod_{k=1}^n
  (1+\ol{\lambda_k}z)$. Note also that the definition of $f^\#$
  depends on $n$. Even if $a_n=0$, so that $f$ is of degree less than
  $n$, it is understood that $f^\#$ is defined as above whenever we take $f\in\C[z]_n$.
\end{Block}

\begin{Block}\label{FischerIP}
  The \emph{Fischer inner product} is the Hermitian inner product on $\C[z]_n$
  given by
  \[
  \scp{f,g}_n = \sum_{k=0}^n \frac{1}{\binom{n}{k}}a_k\ol{b_k}
  \]
for
  $f(z)=\sum_{k=0}^n a_nz^n$ and $g(z)=\sum_{k=0}^nb_nz^n$. The
  following hold for all $f\in\C[z]_n$.
    \begin{enumerate}
    \item\label{FIP:orthogonality} $\scp{z^k,z^l}_n=\delta_{k,l}/\binom{n}{k}$ for
      $k,l=0,\dots,n$;\\[-.8em]
    \item\label{FIP:SharpCheck} $\scp{f,g^{\#\vee}}_n = (-1)^n\scp{g,f^{\#\vee}}_n$ for all $g\in\C[z]_n$; \\[-.8em]
    \item $\scp{zg,f}_n = \frac{1}{n}\scp{g,f'}_{n-1}$ for all $g\in\C[z]_{n-1}$;\\[-.8em]
    \item $f(\lambda)=\scp{f,(1+\ol\lambda z)^n}_n$ for all
      $\lambda\in\C$;\\[-.8em]
    \item\label{FIP:Symmetrization} $\Sym_n(f)(y_1,\dots,y_n) = \scp{f,\prod_{k=1}^n (1+\ol{y_k} z)}_n =
      \scp{f,(\prod_{k=1}^n (z-y_k))^{\#\vee}}$.
    \end{enumerate}

\begin{proof}
  (1)---(4) are checked directly. To verify (5), expanding the product
  $\prod_{k=0}^n (1+\ol{y_k}z) = \sum_{k=0}^n
  \ol{\sigma_k(y_1,\dots,y_n)}z^k$ shows that
  \begin{align*}
    \bigscp{f,\prod\nolimits_{k=1}^n (1+\ol{y_k} z)}_n &= \bigscp{\sum\nolimits_{k=0}^n
      a_kz^k,\sum\nolimits_{k=0}^n \ol{\sigma_k(y_1,\dots,y_n)}z^k}_n\\
    &=\sum\nolimits_{k=0}^n\frac{1}{\binom{n}{k}}a_k\sigma_k(y_1,\dots,y_n)
    =\Sym_n(f).\qedhere
  \end{align*}
\end{proof}
\end{Block}

\begin{Block}\label{Bracket}
  For $f,g\in\C[z]_n$, we define
\[
  [f,g]_n = \scp{f,g^{\#\vee}}_n.
\]
  Note that if $g$ is monic of degree $n$ with zeros
  $\mu_1,\dots,\mu_n$, then
  \[ 
[f,g]_n = \Sym_n(f)(\mu_1,\dots,\mu_n).
  \]
  by \ref{FischerIP}(\ref{FIP:Symmetrization}). The form $[-,-]_n$ is
  bilinear and non-degenerate and, by \ref{FischerIP}(\ref{FIP:SharpCheck}), satisfies
    \[ 
      [f,g]_n = (-1)^n[g,f]_n.
    \]
  If $[f,g]_n = 0$, then
  $f$ and $g$ are called \emph{apolar}. 
\end{Block}

A \emph{circular domain} is an open or closed disk or
halfspace in $\C$, or the complement of any such set. We use $\D$ to
denote the open unit disk.

\begin{Thm}[Walsh]\label{Thm:Walsh} If $D\subset\C$ is a circular
  domain and $f\in\C[z]_n$ is monic without zeros in $D$, then
  $\Sym_n(f)$ has no zeros in $D^n$.
\end{Thm}

Since the complement of a circular domain is again a circular domain,
an equivalent statement of Walsh's theorem is

\begin{Thm*}[Walsh] Let $D\subset\C$ be a circular domain. If
  $f\in\C[z]_n$ is monic with all zeros in $D$ and
  $(y_1,\dots,y_n)\in\C^n$ is a zero of $\Sym_n(f)$, then $y_k\in D$
  for some $k\in\{1,\dots,n\}$.
\end{Thm*}

\noindent See H{\"o}rmander \cite{Ho54} (or \cite{Ma66, RS02,
  Sz22}) for the proof.

\begin{Cor}[Grace's Theorem]\label{Cor:Grace}
  Let $D\subset\C$ be a circular domain and let $f,g\in\C[z]$ be
  apolar polynomials of the same degree. If all zeros of $f$ are
  contained in $D$,
  then $g$ has at least one zero in $D$.
\end{Cor}

\begin{proof} 
  Let $n=\deg(f)=\deg(g)$. Without loss of generality, we may assume
  that $g$ is monic and write $g=\prod_{k=1}^n (z-\mu_k)$. Then
\[
  [f,g]_n = \Sym_n(f)(\mu_1,\dots,\mu_n)
\]
by \ref{Bracket}. By Walsh's theorem, at least one of $\mu_1,\dots,\mu_n$
must lie in $D$, as claimed.
\end{proof}

\section{Polynomial images of circular domains}\label{Sec:PolyImages}

Let $D\subset\C$ be a circular domain and let $q\in\C[z]$ be a
monic polynomial of degree $d$. We want
to understand how Walsh's and Grace's theorem transform under the map
$q\colon\C\to\C$. We will use the following notation:
\[
q_\circ(D) \;=\; \bigl\{u\in\C\st q^{-1}(u)\subset D\bigr\}.
\]

The set $q_\circ(D)$ is semialgebraic and its boundary
is contained in the real algebraic curve $q(\partial D)$. Since the
complement of a circular domain is a circular domain, the set
$q_\circ(D)$ is the complement of the image of a circular domain, namely
\[
q(\C\setminus D) = \C\setminus q_\circ(D)\quad\text{and}\quad
q_\circ(\C\setminus D) = \C\setminus q(D).
\]
\begin{Example}
  Let $q=z^2 + \beta z + \gamma$. Two points $z,w\in\C$ have the same
  image under $q$ if and only if $z+w = -\beta$. This implies
\[
  q^{-1}(q_\circ(D)) = \bigl\{z\in D\st -(z+\beta)\in D\bigr\} =
  D\cap -(D+\beta).
\]
  and $q_\circ(D)$ is the image of that region under $q$. In
  particular, if $D=\D$ is the open unit disk, then $q_\circ(\D)$ is
  non-empty if and only if $|\beta|<2$.

  For example, take $q=z^2+\frac 12z$. The image of the unit circle
  under $q$ is the real quartic curve $\{z=x+iy\st u(x,y)=0\}$ in the
  complex plane, where
\[
  u(x,y)=4x^4 + 4y^4 + 8x^2y^2 - 9x^2 - 9y^2 - 2x + 3.
\]
  The preimage of this curve under $q$ consists of the unit circle and
  the shifted unit circle $\{z\in\C\st |z+\frac 12|=1\}$. The region
  $q_\circ(\D)$ and its preimage are shown in Fig.~\ref{Fig:QuarticCurve}.
  \begin{figure}[h]
    \centering
    \begin{minipage}[c]{0.35\linewidth}
      \begin{center}
        \includegraphics[height=4cm]{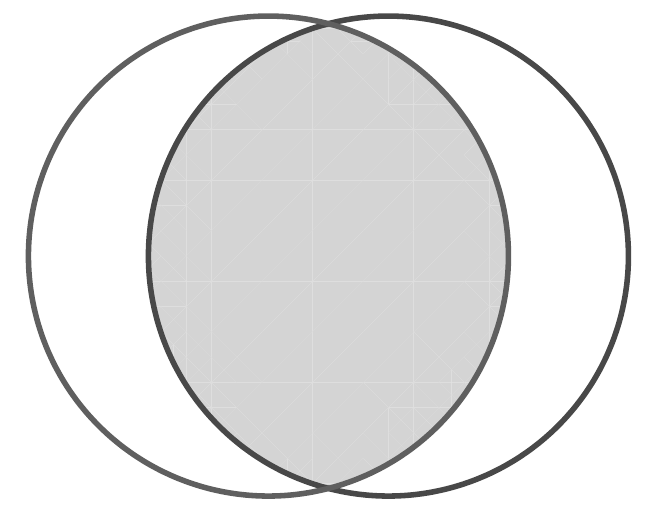}
      \end{center}
    \end{minipage}
    \begin{minipage}[c]{0.1\linewidth}
      \begin{center}
        $\xrightarrow{\ \ \ q\ \ \ }$
      \end{center}
    \end{minipage}
    \begin{minipage}[c]{0.35\linewidth}
      \begin{center}
        \includegraphics[height=5cm]{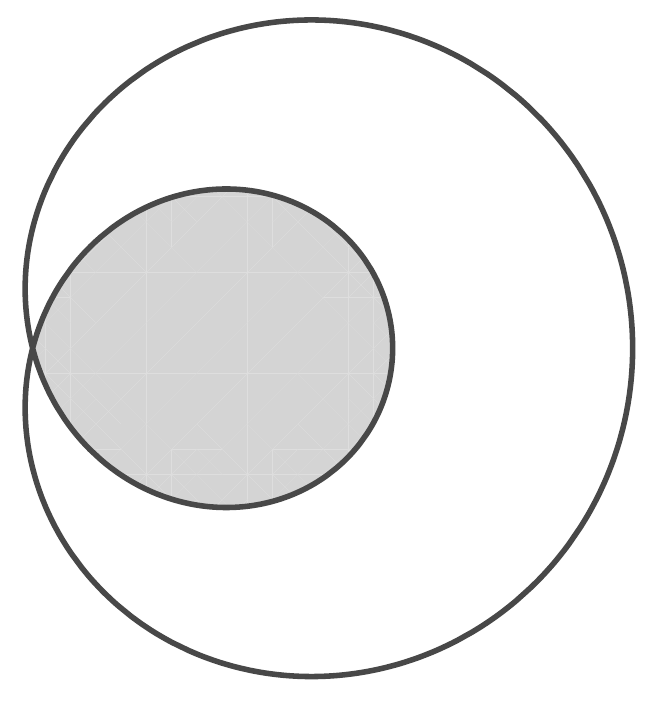}
      \end{center}
    \end{minipage} 

   \caption{The region $q_\circ(\D)$ and its preimage for $q=z^2+\frac
  12 z$}\label{Fig:QuarticCurve}
  \end{figure}
\end{Example}

\begin{Example}\label{Example:Cubic}
  Consider the cubic polynomial $q=z^3 + \frac 12 z$. The image
  of the unit circle under $q$ is the real sextic curve $\{z=x+iy\st
  h(x,y)=0\}$ given by
\[
  u(x,y) = 16 x^6 +48 x^4 y^2 -52 x^4 +48 x^2 y^4 -104 x^2 y^2+40 x^2 +16 y^6-52 y^4+48 y^2-9
\]
shown on the right-hand side of Figure  \ref{Fig:SexticCurve}. The preimage of
this curve under $q$ has two real components, the unit circle and a
curve of degree $8$ given by the vanishing of the polynomial
\begin{align*}
v(x,y) = \;& 16 x^8 +16 x^6 +64 x^6 y^2 +96 x^4 y^4 -16 x^4 y^2 +8 x^4 +64 x^2 y^6 -80 x^2 y^4\\
& -16 x^2 y^2 -28 x^2 +16 y^8 -48 y^6 +40 y^4 -12 y^2 + 9.
\end{align*}
This curve is shown on the left of Figure
\ref{Fig:SexticCurve}. Together with the unit circle, it bounds the
region $q^{-1}(q_{\circ}(\D))$. 

  \begin{figure}[h]
    \centering
    \begin{minipage}[c]{0.35\linewidth}
      \begin{center}
        \includegraphics[height=5cm]{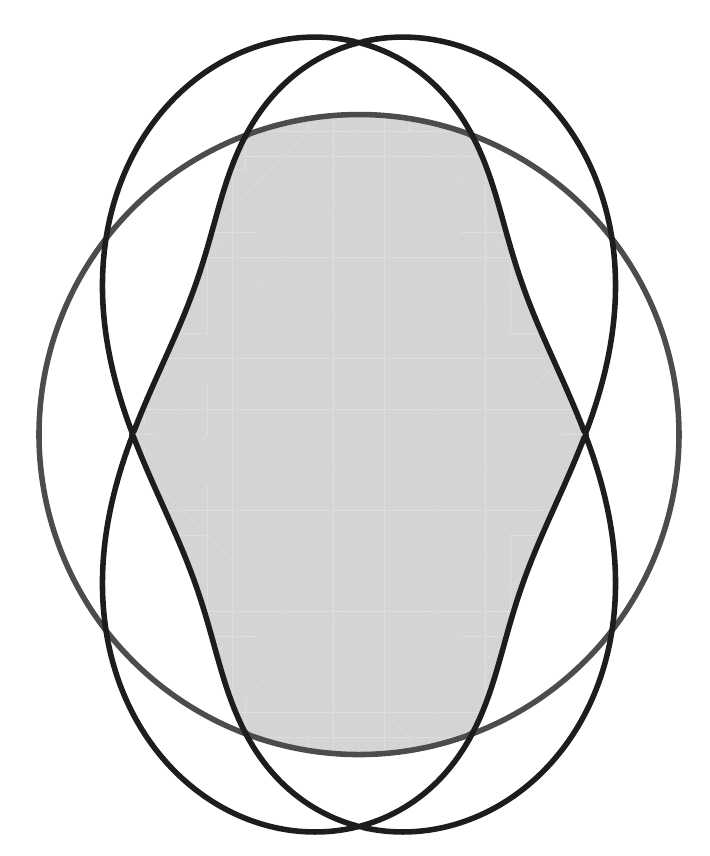}
      \end{center}
    \end{minipage}
    \begin{minipage}[c]{0.1\linewidth}
      \begin{center}
        $\xrightarrow{\ \ \ q\ \ \ }$
      \end{center}
    \end{minipage}
    \begin{minipage}[c]{0.35\linewidth}
      \begin{center}
        \includegraphics[height=5cm]{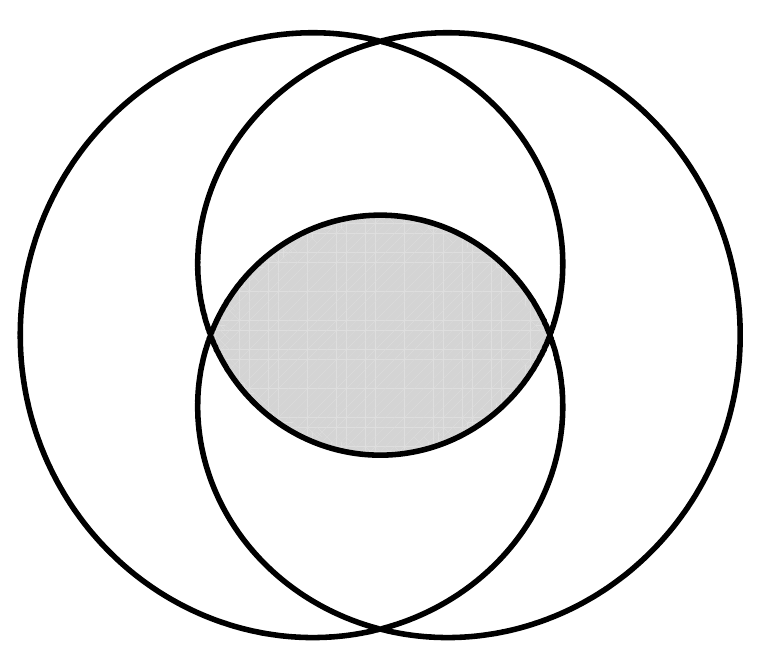}
      \end{center}
    \end{minipage} 

   \caption{The region $q_\circ(\D)$ and its preimage for $q=z^3+\frac
  12 z$}\label{Fig:SexticCurve}
  \end{figure}

\end{Example}

\begin{Remark}
  Note that the region $q_\circ(D)$ need not be connected in
  general. To construct an example in which it is not, consider a
  conformal map $\phi\colon\D\to\Omega$ from the unit disc onto an
  annular sector
\[
   \Omega=\biggl\{z\in\C\:\biggl|\: \frac 12<|z|<1,\ 0<{\rm arg}(z) <\frac{3\pi}{2}\biggr\}.
\]
Such a map $\phi$ exists by the Riemann mapping theorem and, by Caratheodory's
theorem \cite{Ca13}, extends continuously to a map
$\Phi\colon\ol{\D}\to\ol{\Omega}$ between the closures. Now $\Phi$ can
be approximated uniformly by a sequence of polynomials. Consequently,
if $p\in\C[z]$ satisfies
\[
  \|\Phi-p\|_{\infty,\D}<\epsilon,
\]
the image $p(\D)$ has Hausdorff-distance at most $\epsilon$ from
$\Omega$. Put $q=p^2$, then, by construction, the complement of the
image $q(\D)$ is disconnected, the origin being contained in a bounded
connected component of $\C\setminus q(\D)$. Thus if we consider the circular region
$D=\C\setminus\D$, then $q_\circ(D)=\C\setminus q(\D)$ is not
connected.
\end{Remark}

In order to test whether $q_\circ(\D)$ is non-empty for a given $q$, we
may proceed as follows. Given a monic polynomial $q\in\C[z]_d$, write
\[
  \frac{q^\#(x)\ol{q^\#(\ol y)} - q(x)\ol{q(\ol y)}}{1-xy} =
  \sum_{j,k=0}^{d-1} a_{jk} x^j y^k.
\]
By the Schur-Cohn criterion (see \cite[\S 3.3]{KN81}), all roots
of $q$ are contained in $\D$ if and only if the Hermitian form defined
by the $d\times d$-matrix 
\[
{\rm SC}(q) = (a_{jk})_{j,k=0,\dots,d-1}
\]
is positive definite. This proves the following.

\begin{Prop}
  The region $q_\circ(\D)$ is empty if and only if the matrix ${\rm
    SC}(q-\lambda)$ has a non-positive eigenvalue, for every
  $\lambda\in\C$.\qed
\end{Prop}

\noindent Similar criteria exist for the case of a
halfplane instead of a disk. 

\begin{Example}
  Consider the family of cubic polynomials $q(z)=z^3+\gamma z$ for
  $\gamma\in\C$. We compute the Schur-Cohn matrix and find
\[ 
{\rm SC}(q-\lambda) = 
\begin{bmatrix}
  1-|\lambda|^2 & \ol{\lambda}\gamma & \ol{\gamma}\\
\lambda\ol{\gamma} & 1-|\lambda|^2-|\gamma|^2 & \ol{\lambda}\gamma\\
\gamma & \lambda\ol{\gamma} & 1-|\lambda|^2
\end{bmatrix}.
\]
Thus we see that if $|\gamma|\ge 1$, the matrix ${\rm SC}(q-\lambda)$
is not positive definite for any $\lambda$, so that
$q_\circ(\D)=\emptyset$. Conversely, if $|\gamma|<1$, then ${\rm
  SC}(q)$ is positive definite and hence
$q_\circ(\D)\neq\emptyset$. 

For real $\gamma\in (0,1)$, the picture in the complex plane will
essentially look the same as in the case $\gamma=\frac 12$ in Example
\ref{Example:Cubic} above. For $\gamma=1$, the image curve will
degenerate as in Figure \ref{Fig:EmptyRegion} and the region
$q_\circ(\D)$ will be empty.
\begin{figure}[h]
    \centering
    \begin{minipage}[c]{0.35\linewidth}
      \begin{center}
        \includegraphics[height=5cm]{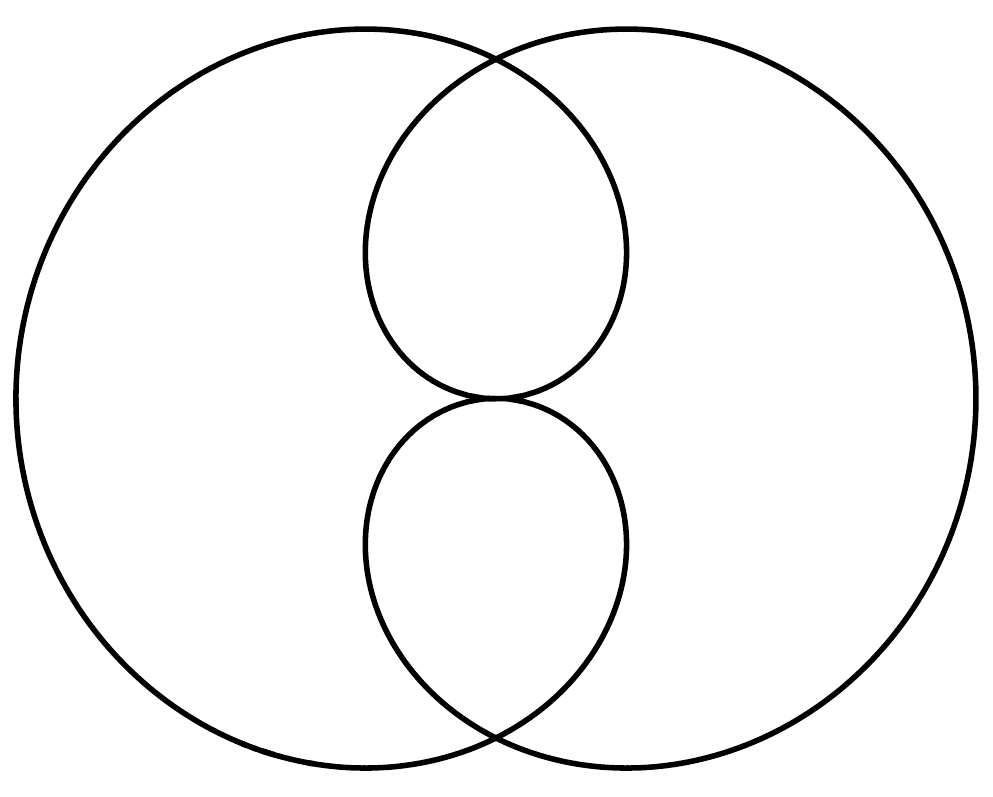}
      \end{center}
    \end{minipage}
   \caption{For $q=z^3+z$, the region $q_\circ(\D)$ is empty.}\label{Fig:EmptyRegion}
  \end{figure}
\end{Example}

\begin{Remark}
  If $q$ is viewed as a rational function on the Riemann
  sphere $\ol\C$, then $q^{-1}(\infty)=\infty$. It follows that if $D$
  is an unbounded domain, so that $\infty$ is contained in the
  closure of $D$ inside $\ol\C$, then $\infty$ is also in the closure
  of $q_\circ(D)$, by continuity. In particular, $q_\circ(D)$ is
  non-empty.\\
  More generally, we may consider
\[
  q_{(k)}(D)=\bigl\{u\in\C\st \#(q^{-1}(u)\cap D)=k\bigr\},
\]
  where $\#(q^{-1}(u)\cap D)$ is the cardinality of the fiber,
  counted with multiplicities. Clearly, the regions
  $q_{(k)}(D)$ are pairwise disjoint for different values of
  $k$ and the boundary of $q_{(k)}(D)$ is contained in the curve
  $q(\partial D)$. Furthermore, we have
\[
  q_{(d)}(D)=q_\circ(D),\quad \bigcup_{k=0}^dq_{(k)}(D)=\C \quad\text{and}\quad\bigcup_{k=1}^dq_{(k)}(D)=q(D).
\]
The regions $q_{(k)}(D)$ can also be characterized in terms of
a fiber-counting integral. For example, let
$D=\C\setminus\D$ be the complement of the unit disk. For $u\in\C$, let
\[
N(u) = \frac{1}{2\pi i}\lim_{R\to\infty}\biggl(\int_{|z|=R}\frac{q'(z)}{q(z)-u}dz -
\int_{|z|=1}\frac{q'(z)}{q(z)-u}dz\biggr).
\]
Then $N(u)$ is the number of preimages of $u$ under $q$ contained in
$D$, counted with multiplicities, so that $q_{(k)}(D) =
\{u\in\C\st N(u)=k\}$. 
\end{Remark}

For our purposes, it would be best if we could have
$q_\circ(D)=q(D)$. Unfortunately, this case does not occur in a
non-trivial way, as the following theorem shows.

\begin{Thm}
  The equality $q_\circ(\D)=q(\D)$ occurs if and only if $q=z^d+c$ for
  some $d\ge 1$ and $c\in\C$.
\end{Thm}

\begin{proof}
  Let $q\in\C[z]$ be monic of degree $d$ and assume
  $q_\circ(\D)=q(\D)$, which is equivalent to
  $q^{-1}(q(\D))=\D$. Write $\T=\partial\D$, then $q^{-1}(q(\T))=\T$
  by continuuity.  We show first that this implies
  $q(\T)=\partial q(\D)$.  Note that $q$ is non-constant and hence
  $q\colon\C\to\C$ is a surjective open map, which implies $\partial
  q(\D)\subset q(\T)$. Suppose this inclusion is strict, which means
  that $q(\T)$ contains interior points of $\ol{q(\D)}$. Since $q(\T)$ is
  a closed curve, it follows that $\ol{q(\D)}\setminus q(\T)$ is
  disconnected.
  On the other hand, $q^{-1}(q(\T))=\T$ implies $q(\D)\cap
  q(\T)=\emptyset$, hence $q(\D)$ is disconnected, a contradiction.

  From this we see that $q(\D)$ is an open subset of $\C$ with
  connected boundary and hence it is simply connected. Fix a point
  $a\in q(\D)$. By the Riemann mapping theorem, there exists a
  biholomorphic map $\phi\colon q(\D)\to\D$ with $\phi(a)=0$. Let
  $B=\phi\circ q\colon\D\to\D$. Since $q(\D)$ is simply connected and
  bounded by the Jordan curve $q(\T)$, the holomorphic map $B$ extends
  continuously to a map $\ol\D\to\ol\D$, by Carath\'eodory's theorem
  \cite{Ca13}. This implies that $B$ has a representation as a
  finite Blaschke product of degree $d$, i.e.~there exists a monic
  polynomial $h\in\C[z]$ of degree $d$ and $\alpha\in\C$ such that
\[
  B(z) = \alpha\cdot \frac{h(z)}{h^\#(z)}
\]
for all $z\in\ol\D$ (see for example \cite[\S 20]{Co95}). Let
$c_1,\dots,c_d\in\C$ be the zeros of $h$, then $B(c_k)=\phi(q(c_k))=0$
implies that $c_1,\dots,c_d$ are also zeros of $q(z)-a$, so that
$q-a=h$. Factoring $\phi$ as $\phi(z)=(z-a)\psi(z)$ for some holomorphic
map $\psi\colon q(\D)\to\C$, we can write
\[
\alpha\cdot\frac{q(z)-a}{q^\#(z)-\ol{a}z^d} =
\alpha\cdot\frac{h(z)}{h^\#(z)} = B(z) = \bigl(q(z)-a\bigr)\psi\bigl(q(z)\bigr)
\]
for $z\in\D$. Dividing both sides by $q(z)-a$ shows
\[
q^\#(z) - \ol{a}z^d = \frac{\alpha}{\psi(q(z))}.
\]
Therefore, $q^\#(z)-\ol{a}z^d$ is a polynomial of
degree at most $d$ that is constant along the fibers of $q$. By Lemma
\ref{Lemma:ConstantOnFibers} below, this implies that there exist
constants $b,c\in\C$ such that $q^\#(z)-\ol{a}z^d = b q + c$. But the
point $a\in q(\D)$ can be chosen arbitrarily, so we obtain constants
$b_a,c_a\in\C$ depending on $a\in q(\D)$ and identities
\[
q^\#(z) - \ol{a}z^d = b_a q + c_a.
\]
Expanding $q(z)=z^d+\sum_{j=0}^{d-1}\alpha_j z^j$ and comparing
leading coefficients on both sides leads to $b_a = \ol{\alpha_0} -
\ol{a}$. After cancelling leading coefficients, we are then left with 
\[
1 + \sum_{j=1}^{d-1}\ol{\alpha_{d-j}}z^j =
(\ol{\alpha_0}-\ol{a})\sum_{j=0}^{d-1}\alpha_jz^j + c_a.
\]
For this to hold for all $a\in q(\D)$, we must have
$\alpha_1=\cdots=\alpha_{d-1}=0$, so that $q(z)-z^d$ is constant, as claimed.
\end{proof}

\begin{Lemma}\label{Lemma:ConstantOnFibers}
  If $q,r\in\C[z]$ are two polynomials of the same degree such that $r$ is
  constant along the fibers of $q$, i.e.~$q(z)=q(w)$ implies
  $r(z)=r(w)$ for all $z,w\in\C$, then there are constants $b,c\in\C$
  such that $r = bq + c$.
\end{Lemma}

\begin{proof}
  If $q,r$ are constant, there is nothing to show. Otherwise, consider
  the algebraic curve $Z=\{(z,w)\in\C^2\st q(z) = q(w)\}\subset\C^2$. By
  hypothesis, the polynomial $r(z)-r(w)\in\C[z,w]$ vanishes
  identically on $Z$. Since $q(z)-q(w)$ is square-free, Hilbert's Nullstellensatz gives an
  identity
\[
  r(z) - r(w) = s(z,w)\bigl(q(z)-q(w)\bigr)
\]
for some $s\in\C[z,w]$. Let $\zeta\in\C$ be any zero of $r$,
then $r(z) = s(z,\zeta)\bigl(q(z)-q(\zeta)\bigr)$. Since
$r$ and $q$ are of the same degree, $s(z,\zeta)$ must have degree $0$
in $z$, so that putting $b=s(1,\zeta)$ and
$c=-s(1,\zeta)q(\zeta)$ yields the desired identity.
\end{proof}

\section{Symmetrization and pullback}\label{Sec:Symmetrization}

Let $f\in\C[z]_n$, $q\in\C[z]_d$, with $q$ monic, and consider
$\Sym_{nd}(f\circ q)(y_{11},\dots,y_{nd})$, the symmetrization of
$f\circ q$ in the $nd$ variables $(y_{jk}\st
j=1,\dots,n,k=1,\dots,d)$.  Let
\[
Q\colon\bracemap{\C^n}{\C^n}{(z_1,\dots,z_n)}{\bigl(q(z_1),\dots,q(z_n)\bigr)}.
\]
Since $q$ has degree $d$, the fibers of $Q$ can be identified with
points in $\C^{nd}$. Let $(u_1,\dots,u_n)\in\C^n$ and put
\[
\bigl(y_{11},\dots,y_{nd}\bigr) = Q^{-1}(u_1,\dots,u_n).
\]
In other words, $y_{j1},\dots,y_{jd}$ are the zeros of the polynomial
$q(z)-u_j$. We compute the restriction of $\Sym_{nd}(f\circ q)$ to
this fiber. By \ref{FischerIP}(\ref{FIP:Symmetrization}), we have
\[
\Sym_{nd}(f\circ q)(y_{11},\dots,y_{nd}) = \bigscp{f\circ q,\prod\nolimits_{j=1}^n\prod\nolimits_{k=1}^d (1+\ol{y_{jk}} z)}_{nd}.
\]
Since $q(z) - u_j = \prod_{k=1}^d (z-y_{jk})$, we find $\prod_{k=1}^d
(1+\ol{y_{jk}} z)=(q(z)-u_j)^{\#\vee}$ by \ref{SharpCheck}, hence
\[
\Sym_{nd}(f\circ q)|_{Q^{-1}(u_1,\dots,u_n)} = \bigscp{f\circ q,\prod\nolimits_{j=1}^n(q-u_j)^{\#\vee}}_{nd}.
\]

\noindent We introduce the following notation.
\begin{align*}
&S_{q,n}(f)(u_1,\dots,u_n)  = \bigscp{f\circ q,\prod\nolimits_{j=1}^n(q-u_j)^{\#\vee}}_{nd}\\
&T_{q,n}(f) = S_{q,n}(f)(z,\dots,z)
\end{align*}

\begin{Prop}\label{Prop:propertiesT}\item{}
  \begin{enumerate}[(a)]
  \item $S_{q,n}(f) = \Sym_n(T_{q,n}(f))$ for all $f\in\C[z]_n$.
  \item $\deg\bigl(T_{q,n}(f)\bigr)=\deg(f)$ for all $f\in\C[z]_n$.
  \item The leading coefficient of $T_{q,n}(z^k)$ is
\[
(-1)^{k(d+1)}\textstyle\binom{n}{k}\bigl/\binom{nd}{kd}.
\]
  These leading coefficients are exactly the eigenvalues of
  $T_{q,n}$. 
  \item The linear operator $T_{q,n}\colon\C[z]_n\to\C[z]_n$ is
    invertible. 
  \end{enumerate}
\end{Prop}

\begin{proof}
(a) By construction, $S_{q,n}(f)$ is symmetric and multiaffine of
degree $n$ in $u_1,\dots,u_n$ and satisfies $S_{q,n}(f)(z,\dots,z)=T_{q,n}(f)$. By the
uniqueness of the symmetrization, it therefore coincides with
$\Sym_n(T_{q,n}(f))$.

(b) and (c) For $k\in\{0,\dots,n\}$, we compute 
\[
S_{q,n}(z^k)(u,\dots,u)=\scp{q^k,\bigl((q-u)^{\#\vee}\bigr)^n}_{nd}.
\]
If $q(z)=\sum_{j=0}^{d}b_jz^j$ (where $b_d=1$), then 
\[
(q-u)^{\#\vee} =
(-1)^d\overline{(b_0-u)}z^d+\sum\nolimits_{j=1}^{d}(-1)^{d-j}\overline{b_j}z^{d-j}.
\]
Since monomials in $z$ of different degree are orthogonal and
every term of $q^k$ has degree at most $kd$ in $z$, we see that
$S_{q,n}(z^k)(u,\dots,u)$ is a polynomial of degree at most $k$ in
$u$. 
Since $T_{q,n}(z^k)=S_{q,n}(z^k)(z,\dots,z)$, it follows that
$T_{q,n}(z^k)$ has degree at most $k$ in $z$.

To find the leading coefficient, isolate all terms of degree $k$ in
$\ol{u}$ on the right-hand side of the inner product above: These are of the form
$\binom{n}{k}(-1)^{k(d+1)} \ol{u}^k z^{kd} (zr(z)+1)$ for some polynomial
$r\in\C[z]$. Since the left-hand side is of degree $kd$ in $z$ with
leading term $z^{kd}$, we have $\scp{q^k,z^{kd}zr(z)}_{nd}=0$. Thus
the coefficient of $u^k$ is found to be equal to
$\binom{n}{k}(-1)^{k(d+1)}\scp{z^{kd},z^{kd}}_{nd}=(-1)^{k(d+1)}\binom{n}{k}\bigl/\binom{nd}{kd}$,
as claimed. 

The equality $\deg(T_{q,n}(f))=\deg(f)$ is equivalent to the fact that
the matrix representing $T_{q,n}$ in the basis $(1,z,\dots,z^n)$ is
upper-triangular. Its diagonal entries and thus its eigenvalues are
exactly the leading coefficients of $T_{q,n}(z^k)$. 

(d) follows from (c).
\end{proof}

\noindent We say that two polynomials $f,g\in\C[z]_n$ are \emph{$q$-apolar}
if
\[
[f\circ q,g\circ q]_{nd}=0.
\]

The notion of $q$-apolarity is strongly related to the operator
$T_{q,n}$, as the following proposition shows.

\begin{Prop}\label{Prop:qApolarity}
We have 
\[
  [f\circ q,g\circ q]_{nd} = [T_{q,n}(f),g]_n
\]
  for all $f,g\in\C[z]_n$.
\end{Prop}

\begin{proof}
  If $\deg(g)=n$, let $g=\beta\prod_{j=1}^n (z-\mu_j)$. Using
  Prop.~\ref{Prop:propertiesT}(a) and \ref{FischerIP}, we find
\begin{align*} 
[f\circ q,g\circ q]_{nd} &= \beta\scp{f\circ q,\prod\nolimits_{j=1}^n
  (q-\mu_j)^{\#\vee}}_{nd} = \beta S_{q,n}(f)(\mu_1,\dots,\mu_n)\\
&= \beta \Sym_n(T_{q,n}(f))(\mu_1,\dots,\mu_n) = [T_{q,n}(f),g]_n.
\end{align*}
If $\deg(g)<n$, the identity also holds, by continuity.
\end{proof}

\begin{Example}
  Let 
\begin{align*}
q(z)&=z^2 + \beta z + \gamma\\
f(z)&=az^2 + bz + c
\end{align*}
Following the above computation, we find
\begin{align*}
S_{q,2}(f)(u_1,u_2) &= au_1u_2+\frac 16\biggl((2a\Delta-b)(u_1+u_2)+ a\Delta^2-2b\Delta \biggr)+c,
\end{align*}
where $\Delta=\beta^2-4\gamma$ is the discriminant of $q$. Hence
\begin{align*}
  T_{q,2}(f)(z) = az^2+\frac 16\biggl((4a\Delta-2b)z+ a\Delta^2-2b\Delta \biggr)+c.
\end{align*}

With respect to the basis $(1,z,z^2)$, the operator $T_{q,2}$ is therefore
represented by the uppper-triangular matrix
\[
\begin{bmatrix}
  1 & -\frac 13\Delta & \frac 16\Delta^2 \\
  0 &- \frac 13 & \frac 23\Delta\\
  0& 0& 1
\end{bmatrix}
\]
For a general cubic polynomial
\[
f(z)=az^3 + bz^2 + cz+d,
\]
direct computation shows
\[
  T_{q,3}(f)(z) = -az^3+\frac{1}{10}\bigl(2b-9a\Delta\bigr)z^2 +
  \frac{1}{10}\bigl(-3a\Delta^2+2b\Delta-2c\bigr)z + \frac{1}{20}\bigl(-a\Delta^3+2b\Delta^2-6c\Delta\bigr)+d.
\]
With respect to the basis $(1,z,z^2,z^3)$, the operator $T_{q,3}$ is therefore
represented by the upper-triangular matrix
\[
\begin{bmatrix}
   1  & -\frac{3}{10}\Delta & \frac{1}{10}\Delta^2 & -\frac{1}{20}\Delta^3 \\
   0  & -\frac 15 & \frac 15\Delta & -\frac{3}{10}\Delta^2\\
   0  & 0         & \frac 15 & -\frac{9}{10}\Delta\\
   0  & 0 & 0 & -1
\end{bmatrix}
\]
\end{Example} 
\begin{Example}
Let
\[
  q(z)=z^3 + \beta z^2 + \gamma z + \delta\\
\]
In the basis $(1,z,z^2)$ the operator
$T_{q,2}\colon\C[z]_2\to\C[z]_2$ is represented by the matrix
\[
\begin{bmatrix}
    1 & \frac{1}{30}\Gamma &
    -\frac{1}{15}\Delta\\
    0 & \frac{1}{10} & -\frac{1}{15}\Gamma\\
    0 & 0 & 1
\end{bmatrix}
\]
where $\Delta =
\beta^2\gamma^2-4\gamma^3-4\beta^3\delta+18\beta\gamma\delta-27\delta^2$
is the discriminant of $q$ and $\Gamma=2\beta^3-9\beta\gamma+27\delta$.

The operator $T_{q,3}$ in the basis $(1,z,z^2,z^3)$ is represented by
the matrix
\[
\begin{bmatrix}
  1 & \frac{1}{28}\Gamma &-\frac{1}{28}\Delta
  & 0\\
  0 &\frac{1}{28} & 0 & -\frac{3}{28}\Delta \\
  0 & 0 & \frac{1}{28} & -\frac{3}{28}\Gamma\\
  0 & 0 & 0 & 1
\end{bmatrix}
\]
\end{Example}

\noindent We are now ready for the generalized versions of the theorems of Grace and
Walsh. 

\begin{Thm}[Generalized Grace theorem] Let $D$ be a circular domain
  and $q\in\C[z]$ a monic polynomial of degree $d$. If
  $f,g\in\C[z]$ are $q$-apolar polynomials of the same degree and
  all zeros of $f$ lie in $q_\circ(D)$, then $g$ has at least one zero
  in $q(D)$.
\end{Thm}

\begin{proof}
  If $y_{11},\dots,y_{nd}$ are the zeros of $g\circ q$, then
  $[f\circ q,g\circ q]_{nd}=\Sym_{nd}(f\circ q)(y_{11},\dots,y_{nd})$
  by \ref{Bracket}. So if $[f\circ q,g\circ q]_{nd}=0$, then
  $(y_{11},\dots,y_{nd})$ is a zero of $\Sym_{nd}(f\circ q)$. Since
  all zeros of $f$ are in $q_\circ(D)$, all zeros of $f\circ q$
  are in $D$, and therefore $y_{jk}\in D$ for some $j,k$ by Walsh's
  theorem (Thm.~\ref{Thm:Walsh}). Hence $q(y_{jk})\in q(D)$ is a zero
  of $g$.
\end{proof}

\begin{Thm}[Generalized Walsh theorem]\label{Thm:GenWalsh} Let $D$ be a circular domain
  and $q\in\C[z]$ a monic polynomial. Let $f\in\C[z]$ be a polynomial
  of degree $n$.
  \begin{enumerate}  
  \item Assume that all zeros of $f$ lie in
  $q_\circ(D)$. Then all zeros of $T_{q,n}(f)$ lie in $q(D)$ and if
  $(u_1,\dots,u_n)\in\C^n$ is a zero of
  $S_{q,n}(f)=\Sym_n(T_{q,n}(f))$, then $u_k\in q(D)$ for some $k$.
  \item Assume that $f$ does not vanish in $q(D)$. Then $T_{q,n}(f)$
    does not vanish in
  $q_\circ(D)$ and $S_{q,n}(f)$ does not vanish in $(q_\circ(D))^n$.
  \end{enumerate}
\end{Thm}

\begin{proof}
  (1) Assume that $S_{q,n}(u_1,\dots,u_n)=0$. By definition, $S_{q,n}(f)$
  is the restriction of $\Sym_{nd}(f\circ q)$ to the fibers of $q$, so
  $\Sym_{nd}(f\circ q)(y_{11},\dots,y_{nd})=0$, where
  $\{y_{j1},\dots,y_{jd}\}=q^{-1}(u_j)$ for $j=1,\dots,n$. Since all
  zeros of $f$ are in $q_\circ(D)$, the zeros of $f\circ q$ are in
  $D$. By Walsh's theorem (Thm.~\ref{Thm:Walsh}), this implies
  $y_{jk}\in D$ for some $j,k$ and hence $u_j=q(y_{jk})\in q(D)$, as
  claimed.

  (2) Apply (1) to the circular domain
  $\C\setminus D$ and use the identities $q(\C\setminus D)=\C\setminus
  q_\circ(D)$ and $q_\circ(\C\setminus D)=\C\setminus q(D)$.
\end{proof}

\begin{Cor}\label{Cor:GeneralizedBernstein}
  Let $f,g\in\C[z]$ be monic of degree $n$ without common zeros in $q(D)$.
  \begin{compactenum}\item 
    Given $y=(y_1,\dots,y_n)\in q_\circ(D)$ with
    $S_{q,n}(g)(y)\neq 0$, there exists
    $z\in q(D)$ with $g(z)\neq 0$ such that
\[
  \frac{S_{q,n}(f)(y)}{S_{q,n}(g)(y)}=\frac{f(z)}{g(z)}.
\]
  \item If $|f/g| \ge 1$ on $q(D)$, then 
\[
  \biggl|\frac{S_{q,n}(f)}{S_{q,n}(g)}\biggr| \ge
  1 \quad\text{on }\bigl(q_\circ(D)\bigr)^n.
\]
  \end{compactenum}
\end{Cor}

\begin{proof}
  (1) Suppose that $\alpha\in\C$ is not in
  $(f/g)(q(D))$. Since $f$ and $g$ have no common zeros in $q(D)$,
  this implies that $f-\alpha g$ does
  not vanish anywhere in $q(D)$. Then
  $S_{q,n}(f-\alpha g)=S_{q,n}(f)-\alpha S_{q,n}(g)$ does not vanish
  anywhere in $(q_\circ(D))^n$ by
  Thm.~\ref{Thm:GenWalsh}(2), so $\alpha$ is not assumed by the rational
  function $S_{q,n}(f)/S_{q,n}(g)$ on $(q_\circ(D))^n$.  
  (2) follows immediately from (1).
\end{proof}

\begin{Example}
  Let $q(z)=z^3+\gamma z$ where $\gamma$ is real and positive
  (c.f.~Example \ref{Example:Cubic}). The region $q_\circ(\D)$ contains
  the disc $B_0(1+\gamma)$, while $q(D)$ is contained in the disc
  $B_0(1-\gamma)$.  Now a weaker version of
  Cor.~\ref{Cor:GeneralizedBernstein} says that given $f,g\in\C[z]_n$
  such that $|(f/g)(z)|\ge 1$ for all $z$ with $|z|>1-\gamma$, we must
  have $|(S_{q,n}(f)/S_{q,n}(g))(y_1,\dots,y_n)|$ for all
  $y_1,\dots,y_n$ with $|y_k|>1+\gamma$ for $k=1,\dots,n$.\\
  For instance, if $a,b,c,d\in\C$ are such that
\[
 |az^3 + bz^2 + cz + d|\ge 1\quad\text{whenever}\quad |z|>1-\gamma
\]
it follows that
\[
\bigl|\bigl(1+\gamma^3/7\bigr)ay_1y_2y_3 +
\bigl(1/84+\gamma^3/7\bigr)b(y_1y_2+y_1y_3+y_2y_3)+ 1/84 c
(y_1+y_2+y_3) + d\bigr|\ge 1
\]
whenever $|y_1|,|y_2|,|y_3|>1+\gamma$. 
\end{Example}

\section{Skew eigenfunctions}\label{Sec:SkewEigenfunctions}

Let $\sH$ be a complex Hilbert space. Recall that a map
$\Phi\colon\sH\to\sH$ is called \emph{antilinear} if $\Phi(x+y)=\Phi(x)+\Phi(y)$
and $\Phi(\alpha x)=\ol{\alpha}\Phi(x)$ hold for all $x,y\in\sH$,
$\alpha\in\C$. If $\scp{\Phi x,\Phi y}=\scp{y,x}$ holds for all $x,y\in\sH$,
then $\Phi$ is called \emph{isometric}.

\begin{Def} Let $\sH$ be a complex Hilbert space.  A map
  $C\colon\sH\to\sH$ is called an \emph{antilinear conjugation} if it
  is antilinear, isometric and satisfies
\[
C^2 = \epsilon\cdot{\rm id}, \text{ where }\epsilon\in\{-1,1\}.
\]
\end{Def}

\begin{Lemma}\label{Lemma:AntilinearIsometry}
  Let $\sH$ be a complex Hilbert space of finite dimension $n$ and let
  $J\colon\sH\to\sH$ be an antilinear isometry with $J^2=\id$. Then there exists an
  orthonormal basis $f_1,\dots,f_n$ such that $Jf_k=f_k$ for
  $k=1,\dots,n$. 
\end{Lemma}

\begin{proof}
  Take any vector $h\in\sH$ with $||h||=1$. Then $g=(h+Jh)/2$ satisfies
  $Jg=g$. If $g\neq 0$, we put $f_1=g$. If $g=0$, this means $Jh=-h$.
  We put $f_1=ih$, so that $Jf_1=Jih=-iJh=ih=f_1$. Now since $J$
  is an isometry, the orthogonal complement $h_1^\perp$ is
  $J$-invariant and the claim follows by induction.
\end{proof}

\begin{Thm}\label{Thm:skew_eigenvalues}
  Let $\sH$ be a complex Hilbert space of finite dimension $n$ equipped
  with an antilinear conjugation $C$. Let $T\colon\sH\to\sH$ be an
  invertible linear map satisfying
\[
  T^\ast=CTC.
\]
\begin{enumerate}
\item There exists an antilinear conjugation $J$ with $J^2={\rm id}$
  that commutes with $|T|$, where $|T|=\sqrt{T^\ast T}$, and satisfies $T=CJ|T|$.
\item There exists an orthonormal basis $f_1,\dots,f_n$ of
$\sH$ such that
\[
Tf_k = \lambda_k Cf_k \quad\text{for }k=1,\dots,n,
\]
where $\lambda_1,\dots,\lambda_n\in\R_+$ are the singular values of $T$.
\end{enumerate}
\end{Thm}

\begin{proof}
  (1) Write $C^2=\epsilon\cdot{\rm id}$ as above and let $T=U|T|$ be
  the polar decomposition of $T$, where $U$ is unitary. Then $T^\ast =
  |T|U^\ast$ and $T=CT^\ast C=(CU^\ast C)(\epsilon CU|T|U^\ast
  C)$. Put $U'=CU^\ast C$ and $P=\epsilon CU|T|U^\ast C$. Since $C$
  and $U$ are both unitary, so is $U'$. Furthermore, since $|T|$ is
  positive definite, we have
\[
\scp{Px,x}=\scp{\epsilon CU|T|U^\ast C x,x} = \scp{CU|T|U^\ast Cx,C^2x} =
\scp{Cx, U|T|U^\ast Cx}>0
\]
  for any $x\neq 0$ in $\sH$, which shows that $P$ is also positive
  definite. By the uniqueness of the polar decomposition, $T=U|T|=U'P$
  implies $U=U'$ and $|T|=P$. The first equality means $U=CU^\ast C$,
  hence $\epsilon CU=U^\ast C$. Put $J = U^\ast C$, then $J^2 =
  \epsilon C UU^\ast C = {\rm id}$. Also, $J$ is antilinear and
  $\scp{Jx,Jy}=\scp{U^\ast Cx, U^\ast Cy}=\scp{Cx,Cy}=\scp{y,x}$,
  hence $J$ is isometric. Thus $J$ is an antilinear conjugation.  It
  also commutes with $|T|$, since $J|T|J = \epsilon CU |T| U^\ast C =
  P=|T|$. 

  (2) Let $\lambda_1,\dots,\lambda_n$ be the singular values of $T$, i.e.~the
  eigenvalues of $|T|$. Applying Lemma \ref{Lemma:AntilinearIsometry} for
  the restriction of $J$ to each eigenspace of $|T|$, we choose an
  orthonormal basis $f_1,\dots,f_n$ of corresponding eigenvectors of
  $|T|$ each of which is fixed under $J$. We then have
\[
  Tf_k = CJ|T|f_k = C|T|f_k  = \lambda_k Cf_k
\]
  for $k=1,\dots,n$, as desired. 
\end{proof}

\noindent We apply the above result in the case $\sH=\C[z]_n$ and
obtain the following statement. 

\begin{Cor}
  If $q\in\C[z]$ is monic and of even degree, there exists a basis
  $f_0,\dots,f_n$ of $\C[z]_n$ which is orthonormal with respect to
  the Fischer inner product and satisfies
\[
[f_j\circ q,f_k\circ q]_{nd} = \lambda_j\delta_{jk},
\]
for all $j,k=0,\dots,n$, where $\lambda_0,\dots,\lambda_n$ are the
singular values of the operator $T_{q,n}$.
\end{Cor}

In other words, the base-polynomials furnished by
Thm.~\ref{Thm:skew_eigenvalues} are both orthonormal and pairwise
$q$-apolar.

\begin{proof}
  Let $d=\deg(q)$ and put
  \[
  C\colon \C[z]_n\to\C[z]_n,\ p\mapsto p^{\#\vee}.
  \]
  and $T=T_{q,n}$. It is easily checked that $C$ is indeed an
  antilinear conjugation with $\epsilon=(-1)^n$. The identity
  $T^\ast=CTC$ in the hypothesis of Thm.~\ref{Thm:skew_eigenvalues}
  then follows from Prop.~\ref{Prop:qApolarity}: For all
  $f,g\in\C[z]_n$, we have
  \begin{align*}
    \scp{Tf,g}_n & = [Tf,\epsilon C g]_n = [f\circ q,(\epsilon
    Cg)\circ q]_{nd} \\
    & = [(\epsilon C g)\circ q, f\circ q]_{nd}=[T(\epsilon C
    g),f]_n = \epsilon[f,T(\epsilon C g)]_n =\\
    & = \scp{f,CTCg}_n,
  \end{align*}
  hence $T^\ast=CTC$, as claimed. (Note that $[-,-]_{nd}$ is
  symmetric, since $d$ is even.)

  Now the $f_0,\dots,f_n$ of $\C[z]_n$ given by
  Thm.~\ref{Thm:skew_eigenvalues}(2) indeed satisfy
  \begin{align*}
    [f_j\circ q,f_k\circ q]_n &= [T f_j, f_k]_n = \scp{Tf_j,Cf_k}_n\\
    &=\scp{\lambda_j Cf_j,Cf_k}_n=\lambda_j\delta_{jk}
  \end{align*}
  by Prop.~\ref{Prop:qApolarity}, since $C$ is
  isometric.
\end{proof}


\begin{bibdiv}
  \begin{biblist}
\bib{Be26}{book}{
    author = {S. {Bernstein}},
    title = {{Le\c{c}ons sur les propri\'et\'es extr\'emales et la meilleure approximation des fonctions analytiques d'une variable r\'eelle.}},
    date = {1926},
    publisher = {{X + 208 p. Paris, Gauthier-Villars. (Collection de monographies sur la th\'eorie des fonctions)}},
}

\bib{Ca13}{article}{
   author={Carath{\'e}odory, C.},
   title={\"Uber die gegenseitige Beziehung der R\"ander bei der konformen
   Abbildung des Inneren einer Jordanschen Kurve auf einen Kreis},
   journal={Math. Ann.},
   volume={73},
   date={1913},
   number={2},
   pages={305--320},
}

\bib{Co95}{book}{
   author={Conway, J. B.},
   title={Functions of one complex variable. II},
   series={Graduate Texts in Mathematics},
   volume={159},
   publisher={Springer-Verlag, New York},
   date={1995},
   pages={xvi+394},
}

\bib{CS35}{article}{
   author={Van der Corput, J. G.},
   author={Schaake, G.},
   title={Ungleichungen f\"ur Polynome und trigonometrische Polynome},
   language={German},
   journal={Compositio Math.},
   volume={2},
   date={1935},
   pages={321--361},
}

  \bib{Di38}{article}{
    author = {J. {Dieudonn\'e}},
    title = {{La th\'eorie analytique des polyn\^omes d'une variable (a coefficients quelconques).}},
    date = {1938},
    journal = {{Mem. Sci. Math. Fasc. 93, Paris: Gauthier-Villars.}},
    pages={1--71},
}

\bib{Fi17}{article}{
    author = {E. {Fischer}},
    title = {{\"Uber die Differentiationsprozesse der Algebra.}},
    journal = {{J. Reine Angew. Math.}},
    volume = {148},
    pages = {1--78},
    date = {1917},
    publisher = {Walter de Gruyter, Berlin},
}

\bib{GPP14}{article}{
   author={Garcia, S. R.},
   author={Prodan, E.},
   author={Putinar, M.},
   title ={Mathematical and physical aspects of complex symmetric
     operators},
   date = {2014},
   journal={J. Phys. A: Math. Gen.},
   volume={47}
}

\bib{GY03}{book}{
   author={Grace, J. H.},
   author={Young, A.},
   title={The algebra of invariants},
   publisher={Cambridge University Press, Cambridge},
   date={1903},
   pages={ii+viii+384},
}

\bib{Ho54}{article}{
   author={H{\"o}rmander, L.},
   title={On a theorem of Grace},
   journal={Math. Scand.},
   volume={2},
   date={1954},
   pages={55--64},
}

   \bib{KN81}{article}{
   author={Kre{\u\i}n, M. G.},
   author={Na{\u\i}mark, M. A.},
   title={The method of symmetric and Hermitian forms in the theory of the
   separation of the roots of algebraic equations},
   note={Translated from the Russian by O. Boshko and J. L. Howland},
   journal={Linear and Multilinear Algebra},
   volume={10},
   date={1981},
   number={4},
   pages={265--308},
}

\bib{Ma66}{book}{
   author={Marden, M.},
   title={Geometry of polynomials},
   series={Second edition. Mathematical Surveys, No. 3},
   publisher={American Mathematical Society, Providence, R.I.},
   date={1966},
   pages={xiii+243},
}

\bib{RS02}{book}{
   author={Rahman, Q. I.},
   author={Schmeisser, G.},
   title={Analytic theory of polynomials},
   series={London Mathematical Society Monographs. New Series},
   volume={26},
   publisher={The Clarendon Press, Oxford University Press, Oxford},
   date={2002},
   pages={xiv+742},
}

\bib{Se14}{article}{
   author={Sendov, B.},
   author={Sendov, H.},
   title={Loci of complex polynomials, part I},
   journal={Trans. Amer. Math. Soc.},
   volume={366},
   date={2014},
   number={10},
   pages={5155--5184},
}
\bib{Sz22}{article}{
   author={Szeg{\"o}, G.},
   title={Bemerkungen zu einem Satz von J. H. Grace \"uber die Wurzeln
   algebraischer Gleichungen},
   language={German},
   journal={Math. Z.},
   volume={13},
   date={1922},
   number={1},
   pages={28--55},
}
	
  \end{biblist}
\end{bibdiv}
\end{document}